\documentclass[10pt]{amsart}
\usepackage{amsmath}
\usepackage{amsfonts}

\usepackage[utf8]{inputenc}
\usepackage[all]{xy}
\usepackage[english]{babel}
\usepackage{amssymb}
\usepackage[lite, initials]{amsrefs}
\newtheorem{teorema}{Theorem}[section]
\newtheorem*{theorem*}{Main Theorem}
\newtheorem{lemma}[teorema]{Lemma}

\theoremstyle{definition}

\newtheorem{rem}{Remark}[section]
\newtheorem{defin}[teorema]{Definition}
\usepackage{enumerate, color}
\usepackage{hyperref}

\def\R{{\mathbb R}}
\def\N{{\mathbb N}}
\def\C{{\mathbb C}}

\def\e{p_}
\def\p{\e}
\def\a{\alpha_}
\def\b{\beta}

\title[Toeplitz Operators and Skew Carleson measures for weighted Bergman spaces]% end with percent
 {Toeplitz Operators and Skew Carleson measures for weighted Bergman spaces on strongly pseudoconvex domains} % This is the full title of the paper
\author[Marco Abate]{Marco Abate$\dagger$}
\address{Marco Abate\\ Dipartimento di Matematica\\ Universit\`a di Pisa\\ Largo Pontecorvo 5, I-56127 Pisa\\ Italy.} \email{marco.abate@unipi.it}

\author[Samuele Mongodi]{Samuele Mongodi}
\address{Samuele Mongodi\\ Politecnico di Milano\\ Dipartimento di Matematica\\ Via Bonardi, 9 I-20133 Milano\\ Italy}
\email{samuele.mongodi@polimi.it}

\author[Jasmin Raissy]{Jasmin Raissy*}
\address{Jasmin Raissy\\ Institut de Math\'ematiques de Toulouse, UMR5219\\ Universit\'e de Toulouse, CNRS\\ UPS, F-31062 Toulouse Cedex 9\\ France}
\email{jraissy@math.univ-toulouse.fr}
\thanks{$\dagger$ Partially supported by PRA grant ``Sistemi dinamici in analisi, geometria, logica e meccanica celeste", University of Pisa.}
\thanks{$^{*}$Partially supported by the FIRB2012 grant ``Differential Geometry and Geometric Function Theory'', RBFR12W1AQ 002.}

\begin{document}

\begin{abstract}
In this paper we study mapping properties of Toeplitz-like operators on weighted Berg\-man spaces of bounded strongly pseudconvex domains in $\C^n$. In particular we prove that a Toeplitz operator built using as kernel a weighted Bergman kernel of weight $\beta$ and integrating against a measure $\mu$ maps continuously (when $\beta$ is large enough) a weighted Bergman space $A^{p_1}_{\alpha_1}(D)$ into a weighted Bergman space $A^{p_2}_{\alpha_2}(D)$ if and only if $\mu$ is a $(\lambda,\gamma)$-skew Carleson measure, where $\lambda=1+\frac{1}{p_1}-\frac{1}{p_2}$ and $\gamma=\frac{1}{\lambda}\left(\beta+\frac{\alpha_1}{p_1}-\frac{\alpha_2}{p_2}\right)$. This theorem generalizes results obtained by Pau and Zhao on the unit ball, and extends and makes more precise results obtained by Abate, Raissy and Saracco on a smaller class of Toeplitz operators on bounded strongly pseudoconvex domains.
\end{abstract}

\maketitle

\section{Introduction}

Carleson measures are a powerful tool and an interesting object to study, introduced by Carleson \cite{C} in his celebrated solution of the corona problem. Let $A$ be a (usually) Banach space of holomorphic functions on a domain $D\subset\C^n$; given $p >0$, a finite positive Borel measure $\mu$ on~$D$ is a \emph{Carleson measure} for~$A$ and~$p$ if there is a continuous inclusion $A\hookrightarrow L^p(\mu)$, that is, if there exists a constant $C>0$ such that
\[
\forall f\in A \qquad \int_D |f|^p\,d\mu\le C\|f\|_A^p\;.
\]
We shall also say that $\mu$ is a \emph{vanishing Carleson measure} for~$A$ and~$p$ if the inclusion $A\hookrightarrow L^p(\mu)$ is compact.

In this paper we are interested in Carleson measures for weighted Bergman spaces~$A^p_\beta(D)$, that is spaces of holomorphic functions on a domain $D\Subset\mathbb{C}^n$ which are $p$-integrable with respect to the measure~$\delta^\beta\nu$, where $\nu$ is the Lebesgue measure, $\delta$ is the Euclidean distance from the boundary of $D$ and $\beta\in\mathbb{R}$; we shall denote by $A^p(D)$ the (unweighted) Bergman space $A^p_0(D)$.

Carleson measures for (possibly weighted) Bergman spaces have been studied by several authors, including Hastings \cite{H}, Oleinik and Pavlov \cite{OP}, Oleinik \cite{O} and Luecking~\cite{Luecking} for the unit disk $\Delta\subset\mathbb{C}$; Cima and Wogen \cite{CW}, Duren and Weir \cite{DW}, Zhu~\cite{Zh} and Kaptano\u glu~\cite{Ka} for the unit ball $B^n\subset\C^n$; Zhu \cite{Zh1} for bounded symmetric domains; Cima and Mercer \cite{CM}, Abate and Saracco \cite{AbaSar}, Abate, Raissy and Saracco \cite{AbaRaiSar}, Hu, Lv and Zhu \cite{HuLvZhu} and Abate and Raissy~\cite{AbaRai} for strongly pseudoconvex domains.

One of the reasons of the interest for Carleson measures is that they can be characterized in several different ways, even without any reference to function spaces. A particularly important characterization relies on the intrinsic Kobayashi geometry of the domain $D\Subset\mathbb{C}^n$. Given $z_0\in D$ and $r\in(0,1)$, let $B_D(z,r)$ denote the Kobayashi ball of $D$ with center $z_0$ and radius $\frac{1}{2}\log\frac{1+r}{1-r}$. If $\mu$ is a finite positive Borel measure on~$D$, for any $r\in(0,1)$ and $\theta\in\mathbb{R}$ we can compare the $\mu$-measure and the Lebesgue measure of the Kobayashi balls by using the functions
\[
\hat\mu_{r,\theta}(z)=\frac{\mu\bigl(B_D(z,r)\bigr)}{\nu\bigl(B_D(z,r)\bigr)^\theta}\;.
\]
It turns out that the behavior of $\hat\mu_{r,\theta}$ can be used to decide whether $\mu$ is Carleson for a given weighted Bergman space. Indeed we have the following statement:

\begin{teorema}[Abate-Raissy-Saracco~\cite{AbaRaiSar}, Hu-Lv-Zhu~\cite{HuLvZhu}]
\label{th:1.1}
Let $D\Subset\C^n$ be a bounded strongly pseudoconvex smooth domain and $\mu$ a finite positive Borel measure on $D$. Choose $0< p$, $q<+\infty$ and $\alpha>-1$, and denote by $\delta\colon D\to\mathbb{R}^+$ the Euclidean distance from the boundary of $D$. Then:
\begin{itemize}
\item[\textrm{(i)}] if $p\le q$, then $\mu$ is a Carleson measure for $A^p_\alpha(D)$ and $q$ if and only if $\hat\mu_{r,q/p}\delta^{-\alpha q/p}\in L^\infty(D)$ for some (and hence any) $r\in(0,1)$;
\item[\textrm{(ii)}] if $p>q$, then $\mu$ is a Carleson measure for $A^p_\alpha(D)$ and $q$ if and only if $\hat\mu_{r,1}\delta^{-\alpha q/p}\in L^{\frac{p}{p-q}}(D)$ for some (and hence any) $r\in(0,1)$.
\end{itemize}
\end{teorema}

In view of this theorem it is natural to say that a measure $\mu$ is a \emph{$(\lambda,\alpha)$-skew Carleson measure} if $\lambda\ge 1$ and $\hat\mu_{r,\lambda}\delta^{-\alpha\lambda}\in L^\infty(D)$, or if $\lambda<1$ and $\hat\mu_{r,1}\delta^{-\alpha\lambda}\in L^{\frac{1}{1-\lambda}}(D)$. When $\lambda=1$ (i.e., $p=q$) we shall say that $\mu$ is a $\alpha$-Carleson measure.

Other characterizations can be given in terms of $r$-lattices and of the Berezin transform of the measure $\mu$ (see Section~2 of this paper for details); but here we are interested in a different kind of characterization, an application of Carleson measures to mapping properties of Toeplitz operators.

Roughly speaking, a Toeplitz operator is the composition of a projection and a multiplication. More precisely, if $X$ is a Banach algebra, $Y\subset X$ a Banach subspace, $P\colon X\to Y$ a linear projection and $f\in X$, then the \emph{Toeplitz operator}~$T_f$ of \emph{symbol}~$f$ is given by $T_f(g)=P(fg)$.

In complex analysis, the most important projection is the \emph{Bergman projection}~$B$, which is the orthogonal projection of the space $L^2(D)$ onto the (unweighted) Bergman space $A^2(D)$, where $D\Subset\mathbb{C}^n$ is a bounded domain. The Bergman projection is an integral operator of the form
\[
Bf(z)=\int_D K(z,w)f(w)\,d\nu(w)\;,
\]
where $K\colon D\times D\to\C$ is the \emph{Bergman kernel} of~$D$. It turns out that the Bergman projection can be extended to $L^p(D)$ for all $p>0$ and maps $L^p(D)$ into $A^p(D)$. \v Cu\v ckovi\'c and McNeal \cite{CMc} suggested to study the mapping properties of Toeplitz operators, associated to the Bergman projection, of the form
\[
T_{\delta^\beta}f(z)=\int_D K(z,w)f(w)\delta(w)^\beta\,d\nu(w)\;;
\]
in particular they were interested in determining for which values of $\beta\in\R$ the operator $T_{\delta^\beta}$ would map a Bergman space $A^p(D)$ into a Bergman space $A^q(D)$. In the paper \cite{AbaRaiSar} we realized that to properly address \v Cu\v ckovi\'c and McNeal's questions it is useful to consider the larger class of Toeplitz operators associated to measures. If $\mu$ is a finite positive Borel measure on $D$ then the \emph{Toeplitz operator} of \emph{symbol} $\mu$ is given by
\[
T_\mu f(z)=\int_D K(z,w)f(w)\,d\mu(w)\;;
\]
clearly, the Toeplitz operator $T_{\delta^\beta}$ considered by \v Cu\v ckovi\'c and McNeal is the Topelitz operator of symbol the measure $\delta^\beta\nu$. Toeplitz operators with a measure as symbol have been studied, for instance, by Kaptano\u glu~\cite{Ka} on the unit ball of $\C^n$, by Li~\cite{Li} and Li and Lueckling~\cite{LiLu} in strongly pseudoconvex domains, and by Schuster and Varolin~\cite{ScVa} in the setting of weighted Bargmann-Fock spaces on $\C^n$; they already noticed relationships between their mapping properties and Carleson properties of $\mu$.

In \cite{AbaRaiSar} we performed a detailed study of how Carleson properties of $\mu$ were related to mapping properties of $T_\mu$, proving results like the following:

\begin{teorema}[Abate-Raissy-Saracco \cite{AbaRaiSar}]
\label{th:1.2}
Let $D\Subset\C^n$ be a bounded strongly pseudoconvex smooth domain, $\mu$ a finite positive Borel measure on $D$ and take $1< p<q<+\infty$. Then the following assertions are equivalent:
\begin{itemize}
\item[\textrm{(i)}] $T_\mu\colon A^p(D)\to A^q(D)$ continuously;
\item[\textrm{(ii)}] $\mu$ is a $\left(1+\frac{1}{p}-\frac{1}{q},0\right)$-skew Carleson measure.
\end{itemize}
\end{teorema}

In proving this theorem we realized that the natural setting to study the mapping properties of Toeplitz operators of this kind is given by weighted Bergman spaces, and we obtained several results showing that if $T_\mu$ maps a weighted Bergman space into another weighted Bergman space then $\mu$ is $(\lambda,\alpha)$-skew Carleson for suitable $\lambda$ and $\alpha$, and conversely that if $\mu$ is $(\lambda,\alpha)$-skew Carleson then $T_\mu$ maps a suitable weighted Bergman space into another suitable weighted Bergman space. Unfortunately, we got only a few clean ``if and only if" statements; moreover, we were mainly interested in mapping spaces $A^p_\alpha(D)$ in spaces $A^q_\beta(D)$ with $q\ge p$, and we did not discuss the case $p>q$.

\smallskip 

This paper is devoted to prove instead a neat and general ``if and only if" statement, following ideas introduced by Pau and Zhao \cite{PZ} in the unit ball.  To do so we proceed by further enlarging the class of Toeplitz operators we are considering. Given $\beta>-1$, the orthogonal projection $P_\beta\colon L^2(\delta^\beta\nu)\to A^2_\beta(D)$ is still represented by an integral operator of the form
\[
P_\beta f(z)=\int_D K_\beta(z,w)f(w)\delta(w)^\beta\,d\nu(w)\;,
\]
where the \emph{weighted Bergman kernel} $K_\beta\colon D\times D\to\C$ has properties similar to those of the usual Bergman kernel (see Section 2). The \emph{Toeplitz operator}~$T_\mu^\beta$ of \emph{symbol}~$\mu$ and \emph{exponent}~$\beta$ is given by
\[
T^\beta_\mu f(z)=\int_D K_\beta(z,w)f(w)\,d\mu(w)\;.
\]
Then the main result of this paper is the following:

\begin{teorema}
\label{th:1.3}
Let $D\Subset\C^n$ be a bounded strongly pseudoconvex smooth domain.
Let $0<p_1$,~$p_2<+\infty$ and $-1<\alpha_1$,~$\alpha_2<+\infty$. Suppose that $\beta\in\R$ satisfies
\[
n+1+\beta>n\max\left\{1,\dfrac{1}{p_j}\right\}+\frac{1+\alpha_j}{p_j}
\]
for $j=1$, $2$.
Put
$$
\lambda=1+\frac{1}{p_1}-\frac{1}{p_2}
$$
and, if $\lambda\ne 0$, put
$$
\gamma=\frac{1}{\lambda}\left(\beta+\frac{\alpha_1}{p_1}-\frac{\alpha_2}{p_2}\right)\;.
$$
Then, for any finite positive Borel measure $\mu$ on $D$, the following statements are equivalent:
\begin{itemize}
\item[\textrm{(i)}] $T_\mu^\beta\colon A^{p_1}_{\alpha_1}(D)\to A^{p_2}_{\alpha_2}(D)$ continuously;
\item[\textrm{(ii)}] the measure $\mu$ is a $(\lambda,\gamma)$-skew Carleson measure.
\end{itemize}
\end{teorema}

In particular, Theorem~\ref{th:1.2} is now obtained as a consequence of Theorem~\ref{th:1.3} by taking $\alpha_1=\alpha_2=\beta=0$ and $1<p_1<p_2<+\infty$.

\smallskip
The paper is structured as follows. In Section~2 we collect a number of preliminary results, on the Kobayashi geometry of strongly pseudoconvex domains, on the weighted Bergman kernels, and on the known characterizations of skew Carleson measures. Section~3 is devoted to the proof of Theorem~\ref{th:1.3}, while in Section~4 we prove a version of Theorem~\ref{th:1.3} for vanishing skew Carleson measures, showing that (under the same hypotheses on the parameters) $T_\mu^\beta\colon A^{p_1}_{\alpha_1}(D)\to A^{p_2}_{\alpha_2}(D)$ is compact if and only if the measure $\mu$ is a vanishing $(\lambda,\gamma)$-skew Carleson measure.

\section{Preliminary results}

In this section we collect definitions and preliminary results that we shall use in the rest of the paper.

%%%
From now on, $D\Subset\C^n$ will be a bounded strongly pseudoconvex domain in $\C^n$ with smooth $C^\infty$ boundary. Furthermore, we shall
use the following notations:
\begin{itemize}
\item $\delta\colon D\to\R^+$ will denote the Euclidean distance from the boundary of $D$,
that is $\delta(z)=d(z,\partial  D)$;
\item given two non-negative functions $f$,~$g\colon D\to\R^+$ we shall write $f\preceq g$
to say that there is $C>0$ such that $f(z)\le C g(z)$ for all $z\in D$ (the constant $C$ is
independent of~$z\in D$, but it might depend on other parameters, such as $r$, $\theta$, etc.);
\item given two strictly positive functions $f$,~$g\colon D\to\R^+$ we shall write $f\approx g$
if $f\preceq g$ and $g\preceq f$, that is
if there is $C>0$ such that $C^{-1} g(z)\le f(z)\le C g(z)$ for all $z\in D$;
\item $\nu$ will be the Lebesgue measure;
\item $\mathcal{O}(D)$ will denote the space of holomorphic functions on~$D$, endowed with the topology of uniform convergence on compact subsets;
\item given $0< p< +\infty$, the \emph{Bergman space} $A^p(D)$ is the (Banach if $p\ge 1$) space
$L^p(D)\cap\mathcal{O}(D)$, endowed with the $L^p$-norm;
\item more generally, if $\mu$ is a positive finite Borel measure on~$D$ and $0<p<+\infty$ we shall denote by $L^p(\mu)$ the set of complex-valued $\mu$-measurable functions $f\colon D\to\C$ such that
\[
\|f\|_{p,\mu}:=\left[\int_D |f(z)|^p\,d\mu(z)\right]^{1/p}<+\infty\;;
\]
\item if $\alpha>-1$ we shall write $\nu_\alpha=\delta^\alpha\nu$, we shall denote by $A^p_\alpha(D)$
the \emph{weighted Bergman space}
\[
A^p_\alpha(D)=L^p(\delta^\alpha \nu)\cap\mathcal{O}(D)\;,
\]
and we shall write $\|\cdot\|_{p,\alpha}$ instead of $\|\cdot\|_{p,\delta^\alpha\nu}$;
\item $K\colon D\times D\to\C$ will be the Bergman kernel of~$D$, and %$K_\beta\colon D\times D\to\C$ will be the weighted Bergman kernel of~$D$ for $\beta\in \R$;
for each $z_0\in D$ we shall denote by $k_{z_0}\colon D\to\C$ the \emph{normalized
Bergman kernel} defined by
\[
k_{z_0}(z)=\frac{K(z,z_0)}{\sqrt{K(z_0,z_0)}}=\frac{K(z,z_0)}{\|K(\cdot,z_0)\|_2}\;;
\]
%and, for $\beta\in \R$, we denote by $k_{\beta,z_0}\colon D\to\C$ the \emph{normalized
%weighted Bergman kernel} defined by
%\[
%k_{\beta, z_0}(z)=\frac{K_\beta(z,z_0)}{\sqrt{K_\beta(z_0,z_0)}}=\frac{K_\beta(z,z_0)}{\|K_\beta(\cdot,z_0)\|_2}\;;
%\]
\item given $r\in(0,1)$ and $z_0\in D$, we shall denote by $B_D(z_0,r)$ the Kobayashi ball
of center~$z_0$ and radius $\frac{1}{2}\log\frac{1+r}{1-r}$.
\end{itemize}

\noindent We refer to, e.g., \cites{A,A1,JP,K}, for definitions, basic properties and applications to geometric function theory of the Kobayashi distance; and to \cites{Ho,Ho1,Kr,R} for definitions and basic properties of the Bergman kernel.

\smallskip
Let us now recall a few results we shall need on the Kobayashi geometry of strongly pseudoconvex domains.

\begin{lemma}[\textbf{\cite{AbaSar}*{Lemma 2.2}}]
\label{sette}
Let $D\Subset\C^n$ be a bounded  strongly pseudoconvex domain. Then there is $C>0$ such that
\[
 \frac{1-r}{C}\delta(z_0)\le  \delta(z) \le \frac{C}{1-r}\delta(z_0)
\]
for all $r\in(0,1)$, $z_0\in D$ and $z\in B_D(z_0,r)$.
\end{lemma}

\begin{lemma}
\label{sei}
Let $D\Subset\C^n$ be a bounded strongly pseudoconvex domain, $\beta\in\mathbb{R}$ and $r\in(0,1)$. Then
\[
\nu_\beta\bigl(B_D(\cdot,r)\bigr)\approx \delta^{n+1+\beta}\;,
\]
where the constant depends on~$r$.
\end{lemma}

\begin{proof}
For $\beta=0$ the result can be found in \cite{Li}*{Corollary 7} and \cite{AbaSar}*{Lemma 2.1}. If $\beta\ne 0$ Lemma~\ref{sette} yields
\[
\nu_\beta\bigl(B_D(z_0,r)\bigr)=\int_{B_D(z_0,r)}\delta(z)^\beta\,d\nu(z)\approx \delta(z_0)^\beta \nu\bigl(B_D(z_0,r)\bigr)
\]
and we are done.
\end{proof}

We shall also need the existence of suitable coverings by Kobayashi balls. Recall that for a bounded domain $D\Subset\C^n$, given $r>0$, a \emph{$r$-lattice} in~$D$ is a sequence $\{a_k\}\subset D$ such that $D=\bigcup_{k} B_D(a_k,r)$ and  there exists $m>0$ such that any point in~$D$ belongs to at most $m$ balls of the form~$B_D(a_k,R)$, where $R=\frac{1}{2}(1+r)$.

The existence of $r$-lattices in bounded strongly pseudoconvex domains is ensured by
the following result:

\begin{lemma}[\textbf{\cite{AbaSar}*{Lemma 2.5}}]
\label{uno}
Let $D\Subset\C^n$ be a bounded strongly pseudoconvex domain. Then for every $r\in(0,1)$ there exists an $r$-lattice in~$D$.%, that is there exist
% $m\in{\bf N}$ and a sequence
%$\{a_k\}\subset D$ of points such that
%$D=\bigcup_{k=0}^\infty B_D(a_k,r)$
%and no point of $D$ belongs to more than $m$ of the balls $B_D(a_k,R)$,
%where $R={\frac12}(1+r)$.
\end{lemma}

We shall use a submean estimate for nonnegative plurisubharmonic functions on Kobayashi balls:

\begin{lemma}[\textbf{\cite{AbaSar}*{Corollaries 2.7 and  2.8}}]
 \label{due}
 Let $D\Subset\C^n$ be a bounded strongly pseudoconvex domain. Given $r\in(0,1)$, set $R={\frac12}(1+r)\in(0,1)$. Then there exists a constant $K_r>0$ depending on~$r$ such that
 \[
 \forall{z_0\in D}\ \ \ \ \qquad \chi(z_0)\le
{\frac{K_r}{\nu\left(B_D(z_0,r)\right)}}\int_{B_D(z_0,r)}\chi\,d\nu
 \]
 and
\[
\forall{z_0\in D\ \ \ \forall z\in B_D(z_0,r)}\ \ \ \ \chi(z)\le
{\frac{K_r}{\nu\left(B_D(z_0,r)\right)}}\int_{B_D(z_0,R)}\chi\,d\nu
\]
for every nonnegative plurisubharmonic function $\chi\colon D\to\R^+$.
\end{lemma}

Now we collect a few results on the weighted Bergman kernels. Given $\beta>-1$, the \emph{weighted Bergman projection} is the orthogonal projection $P_\beta\colon L^2(\nu_\beta)\to A^2_\beta(D)$, where $\nu_\beta=\delta^\beta\nu$. It is known (see, e.g., \cite{Eng}), that there exists a function $K_\beta\colon D\times D\to\C$ such that
$$
P_\beta f(z)
=
\int_D K_\beta(z,w)f(w)\delta(w)^\beta d\nu(w)
$$
for all $f\in L^2(\nu_\beta)$. Moreover, $K_\beta(z,w)$ is holomorphic in $z$, we have $K_\beta(w,z)=\overline{K_\beta(z,w)}$ and
$$
f(z)=\int_D K_\beta(z,w)f(w)\delta(w)^\beta d\nu(w)
$$
for all $f\in A^2_\beta(D)$. The function $K_\beta$ is called the \emph{weighted Bergman kernel} of $D$. For $a\in D$, the \emph{normalized weighted Bergman kernel} of $D$ is
$$
k_{\beta,a}(z)=\dfrac{K_\beta(z,a)}{\sqrt{K_\beta(a,a)}}.
$$
When $\beta=0$ we recover the usual Bergman kernel, and we shall write $K$, respectively $k_a$, instead of~$K_0$, respectively $k_{0,a}$.

We shall need a few estimates on the behaviour of the weighted Bergman kernel. They are analogous to the classical estimates for the Bergman kernel and follow from the results obtained by Engli{\v{s}}~\cite{Eng} on the asymptotic behaviour of the weighted Bergman kernel. The first one is the following.

\begin{lemma}
\label{BKbasic}
Let $D\Subset\C^n$ be a bounded strongly
pseudoconvex domain and let $\beta>-1$. Then
\[
\|K_\beta(\cdot,z_0)\|_{2,\beta}=\sqrt{K_\beta(z_0,z_0)}\approx \delta(z_0)^{-(n+1+\beta)/2}\qquad
\hbox{and}\qquad \|k_{\beta,z_0}\|_{2,\beta}\equiv 1
\]
for all $z_0\in D$.
\end{lemma}

\begin{proof}
The first equality, and hence the result for $k_{\beta,z_0}$, is well-known, as well as the whole statement for $\beta=0$ (see, e.g., \cite{Ho1}).

If $\beta\ne 0$, then thanks to the results in \cite{Eng}, the weighted Bergman kernel is smooth outside the boundary diagonal; so, in particular, if $z_0\in D'\Subset D$ the norm   $\|K_\beta(\cdot,z_0)\|_{p,\beta}$ is bounded by a constant depending only on $D'$, $p$ and $\beta$.

Therefore, we only have to estimate the boundary behaviour. Let $q\in \partial D$ and let $U$ be a neighbourhood of $q$ with coordinates $(z',z_n)=(z_1,\ldots, z_n)$ centered in~$q$ such that
$$
D\cap U=\{(z',z_n)\in U\ :\ \mathsf{Re}(z_n)>\psi(z')\}
$$
where $-\psi$ is strongly plurisubharmonic with $\nabla \psi\neq 0$. Set $r(z)=\mathsf{Re}(z_n)-\psi(z')$. We consider an almost-sesquianalitic extension of $r(z)$ on $U\times U$, i.e., a function, which we denote again by $r$, such that:
\begin{itemize}
\item $r(z,w)=\overline{r(w,z)}$,
\item the first derivatives of $r$ with respect to $\bar{z}$ and $w$ vanish at infinite order along $z=w$,
\item $r(z,z)=r(z)$.
\end{itemize}
It easily follows from these properties that
$$
\frac{\partial}{\partial z_j} r(O)=\frac{\partial}{\partial z_j}r(O,O)\;,
$$
and similarly for the other derivatives. Therefore we have
$$
|2r(z,w)-r(z) -r(w)|=c_1|z_n-w_n| + \sum_{j=1}^{n-1}c_2^j|z_j-w_j|^2+ O(\|z-w\|^3).
$$
Moreover $|2r(z,w)-r(z)-r(w)|$ is positive outside $z=w$, and so $c_1>0$. Therefore in a neighbourhood of $(O,O)$ we have that
$$
|r(z,w)| \approx \left(r(z)+r(w)+|z_n-w_n|+\sum_{j=1}^{n-1}|z_j-w_j|^2\right)\;.
$$

The results in \cite{Eng} imply that $K_\beta(z,w)$ is asymptotic to $c(z,w)r(z,w)^{-n-1-\beta}$ for a suitable function $c\in\mathcal{C}^\infty(\overline{D}\times\overline{D})$. Therefore on $U$ we have
$$
|K_\beta(z,w)|\approx \left(r(z)+r(w)+|z_n-w_n|+\sum_{j=1}^{n-1}|z_j-w_j|^2\right)^{-n-1-\beta}\;.$$
Thus, following the same proof as in the classical case, we obtain the assertion.
\end{proof}

A similar estimate, but with uniform constants on Kobayashi balls, is the following.

\begin{lemma}
\label{piu}
Let $D\Subset\C^n$ be a bounded strongly
pseudoconvex domain and let $\beta>-1$. Then for every $r\in(0,1)$ there exist $c_r>0$ and
$\delta_r>0$ such that
if $z_0\in D$ satisfies $\delta(z_0)<\delta_r$ then
\[
\frac{c_r}{\delta(z_0)^{n+1+\beta}}\le |K_\beta(z,z_0)|\le \frac{1}{c_r\delta(z_0)^{n+1+\beta}}
\]
and
\[
\frac{c_r}{\delta(z_0)^{n+1+\beta}}\le |k_{\beta,z_0}(z)|^2\le \frac{1}{c_r\delta(z_0)^{n+1+\beta}}
\]
for all $z\in B_D(z_0,r)$.
\end{lemma}

\begin{proof}
If $\beta=0$ then this is proven in \cite{Li}*{Theorem~12} and \cite{AbaSar}*{Lemma~3.2 and Corollary 3.3}. If $\beta\ne 0$, then thanks to the results in \cite{Eng}, we have that
\begin{equation}
K_\beta(z,z_0)\approx c(z,z_0)\left(r(z)+r(z_0)+|z_n-z_{0,n}|+\sum_{j=1}^{n-1}|z_j-z_{0,j}|^2\right)^{-(n+1+\beta)}
\label{eq:Eng}
\end{equation}
in suitable local coordinates around a point of the boundary diagonal, i.e., if $d(z_0, \partial D)$, $d(z,\partial D)$ and $\|z-z_0\|$ are small enough. By the completeness of the Kobayashi metric, there exists $\delta_r>0$ such that every $z\in B_D(z_0,r)$ satisfies such condition if $\delta(z_0)<\delta_r$. The assertion then follows by arguing as in \cite{Li}*{Theorem~12} or as in \cite{AbaSar}*{Lemma~3.2 and Corollary 3.3}.
\end{proof}

\begin{rem}
Note that in the previous lemma the estimates from above hold even when $\delta(z_0)\ge\delta_r$, possibly with a different constant $c_r$. Indeed, when $\delta(z_0)\ge\delta_r$ and $z\in B_D(z_0,r)$ by Lemma~\ref{sette} there is $\tilde\delta_r>0$ such that $\delta(z)\ge \tilde\delta_r$; as a consequence we can find $M_r>0$ such that $|K_\beta(z,z_0)|\le M_r$ as soon as $\delta(z_0)\ge\delta_r$ and $z\in B_D(z_0,r)$, and the assertion follows from the fact that $D$ is a bounded domain.
\end{rem}

A very useful integral estimate generalizing the analogous ones  for the unweighted Bergman kernel (see \cite{Li}*{Corollary~11, Theorem~13} and \cite{AbaRaiSar}*{Theorem~2.7})
is the following:

\begin{teorema}\label{teo_stima}
Let $D\Subset\C^n$ be a bounded strongly pseudoconvex domain, $z_0\in D$ and $\alpha$,~$\beta>-1$. Then
%\begin{itemize}
%\item[(i)]
for $0<p<+\infty$ and $\alpha-\beta<(n+\beta+1)(p-1)$ we have
$$
\int_D|K_\beta(\zeta, z_0)|^p\delta(\zeta)^\alpha d\nu(\zeta)
\preceq
\delta(z_0)^{\alpha-\beta-(n+\beta+1)(p-1)}\;.
$$
In particular,
\[
\|K(\cdot,z_0)\|_{p,\alpha} \preceq \delta(z_0)^{\frac{n+1+\alpha}{p}-(n+1+\beta)}\;.
\]
%\item[(ii)] for $1\le p <+\infty$ and $\alpha-\beta>(n+\beta+1)(p-1)$ we have
%$$
%\int_D|K_\beta(\zeta, z_0)|^p\delta(\zeta)^\alpha d\nu(\zeta)
%\preceq
%1.
%$$
%\end{itemize}
\end{teorema}

\begin{proof}
If $\beta=0$ then this is proven in \cite{Li}*{Corollary~11, Theorem~13} and \cite{AbaRaiSar}*{Theorem~2.7}. If $\beta\ne 0$, then it suffices to use \eqref{eq:Eng} and follow the same proof as in the unweighted case.
\end{proof}

Finally, the normalized Bergman kernel can be used to build functions belonging to suitable weighted Bergman spaces:

\begin{lemma}
\label{th:fact4}
Let $D\Subset\C^n$ be a bounded strongly pseudoconvex domain, and $\beta>-1$. Given $0<p<+\infty$ and $-1<\alpha<\min\{(n+\beta+1)p-np-1,(n+\beta+1)p-n-1\}$, set
\[
\tau=\frac{n+1+\beta}{2}-\frac{n+1+\alpha}{p}\;.
\]
For each $a\in D$ set $f_a=\delta(a)^\tau k_{\beta,a}$. Let $\{a_k\}$ be an $r$-lattice and $\mathbf{c}=\{c_k\}\in\ell^p$, and put
\[
f=\sum_{k=0}^\infty c_kf_{a_k}\;.
\]
Then $f\in A^p_\alpha(D)$ with $\|f\|_{p,\alpha}\preceq \|\mathbf{c}\|_p$.
\end{lemma}

\begin{proof}
If $\beta=0$ then this is a consequence of {\cite{HuLvZhu}*{Lemma~2.6}}. If $\beta\ne 0$, then it suffices to use the estimates given by Theorem \ref{teo_stima} and follow the same proof as in the unweighted case.
\end{proof}

We also need to recall a few definitions and results about Carleson measures.

\begin{defin}
\label{def:Carluno}
Let $0<p$, $q<+\infty$ and $\alpha>-1$. A \emph{$(p,q;\alpha)$-skew Carleson measure} is a
finite positive Borel measure $\mu$ such that
\[
\int_D |f(z)|^q\,d\mu(z)\preceq \|f\|_{p,\alpha}^q
\]
for all $f\in A^p_\alpha(D)$. In other words, $\mu$ is $(p,q;\alpha)$-skew Carleson if $A^p_\alpha(D)\hookrightarrow L^q(\mu)$ continuously. In this case we shall denote by $\|\mu\|_{p,q;\alpha}$ the operator norm of the inclusion $A^p_\alpha(D)\hookrightarrow L^q(\mu)$.
Furthermore, a $(p,q;\alpha)$-skew Carleson measure is \emph{vanishing} if
\[
\lim_{j\to+\infty} \int_D |f_j(z)|^q\,d\mu(z)=0
\]
for any bounded sequence $\{f_j\}_{j\in\N}\subset A^p_\alpha(D)$ converging to $0$ uniformly on any compact subset of $D$. For $p\ge 1$, $\mu$ is a vanishing $(p,q;\alpha)$-skew Carleson if and only if $A^p_\alpha(D)\hookrightarrow L^q(\mu)$ compactly (see, e.g., \cite{AbaRaiSar}*{Lemma~4.5}).
\end{defin}

\begin{rem}
When $p=q$ we recover the usual (non-skew) notion of Carleson measure for $A^p_\alpha(D)$.
\end{rem}

\begin{defin}
\label{def:Carldue}
Let $\theta\in\R$, and let $\mu$ be a finite positive Borel measure on~$D$. Given $r\in(0,1)$, let
$\hat\mu_{r,\theta}\colon D\to\R$ be defined by
\[
\hat\mu_{r,\theta}(z)=\frac{\mu\bigl(B_D(z,r)\bigr)}{\nu\bigl(B_D(z,r)\bigr)^\theta}\;;
\]
we shall write $\hat\mu_r$ for $\hat\mu_{r,1}$.

We say that $\mu$ is a \emph{geometric $\theta$-Carleson measure} if $\hat\mu_{r,\theta}\in L^\infty(D)$ for all $r\in(0,1)$, that is if
for every $r>0$ we have
\[
\mu\bigl(B_D(z,r)\bigr)\preceq \nu\bigl(B_D(z,r)\bigr)^\theta
\]
for all $z\in D$, where the constant depends only on~$r$.

Furthermore, we shall say that $\mu$ is a \emph{geometric vanishing $\theta$-Carleson measure} if
\[
\lim_{z\to\partial D} \hat\mu_{r,\theta}(z)=0
\]
for all $r\in(0,1)$.
\end{defin}

Notice that Lemma~\ref{sei} yields
\begin{equation}
\hat\mu_{r,\theta}\approx \delta^{-(n+1)(\theta-1)}\hat\mu_r\;.
\label{eq:hat}
\end{equation}

In \cite{AbaRaiSar} we proved (among other things) that, if $p\ge 1$, a measure $\mu$ is $(p,p;\alpha)$-skew Carleson if and only if it is geometric $\theta$-Carleson, where $\theta=1+\frac{\alpha}{n+1}$.
Hu, Lv and Zhu in \cite{HuLvZhu} have given a similar geometric characterization of $(p,q;\alpha)$-skew Carleson measures for all values of $p$ and $q$; to recall their results we need another definition.

 \begin{defin}
 \label{def:Berez}
 Let $\mu$ be a finite positive Borel measure on~$D$, and $s>0$. The \emph{Berezin transform} of \emph{level}~$s$ of~$\mu$ is the function $B^s\mu\colon D\to\R^+\cup\{+\infty\}$ given by
 \[
 B^s\mu(z)=\int_D |k_z(w)|^s\,d\mu(w)\;.
 \]
 \end{defin}

The geometric characterization of $(p,q;\alpha)$-skew Carleson measures is different according to whether $p\le q$ or $p>q$. We first recall the characterization for the case $p\le q$.%, adding to the result in \cite{HuLvZhu} one more item that we shall need later on.

\begin{teorema}[{\cite{HuLvZhu}*{Theorem~3.1}, \cite{AbaRai}*{Theorem~2.15}}]
\label{carthetaCarluno}
Let $D\Subset\C^n$ be a bounded strongly
pseudoconvex domain. Let $0< p\le q<+\infty$ and $\alpha>-1$; set $\theta=1+\frac{\alpha}{n+1}$.
Then the following assertions are
equivalent:
\begin{itemize}
\item[(i)] $\mu$ is a $(p,q;\alpha)$-skew Carleson measure;
\item[(ii)] $\mu$ is a geometric $\frac{q}{p}\theta$-Carleson measure;
\item[(iii)] there exists $r_0\in(0,1)$ such that $\hat\mu_{r_0,\frac{q}{p}\theta}\in L^\infty(D)$;
\item[(iv)] for some (and hence any) $r\in(0,1)$ we have $\hat\mu_{r,\frac{q}{p}}\delta^{-\alpha\frac{q}{p}}\in L^\infty(D)$;
%$\mu\bigl(B_D(\cdot,r_0)\bigr)\preceq \nu\bigl(B_D(\cdot,r_0)\bigr)^{\frac{q}{p}\theta}$;
\item[(v)] for some (and hence any) $r\in(0,1)$ and some (and hence any) $r$-lattice $\{a_k\}$ in $D$ we have
\[
\forall k\in\N\qquad \mu\bigl(B_D(a_k,r)\bigr)\preceq \nu\bigl(B_D(a_k,r)\bigr)^{\frac{q}{p}\theta}\;;
\]
%\item[(v)] there exists $r_0\in(0,1)$ and a $r_0$-lattice $\{a_k\}$ in $D$ such that
%\[
%\mu\bigl(B_D(a_k,r_0)\bigr)\preceq\nu\bigl(B_D(a_k,r_0)\bigr)^{\frac{q}{p}\theta}\;;
%\]
\item[(vi)] for some (and hence all) $s>\theta\frac{q}{p}$ we have
\[
B^s\mu\preceq \delta^{(n+1)\left(\theta\frac{q}{p}-\frac{s}{2}\right)}\;;
\]
%\item[(vii)] there exists $C>0$ such that for some (and hence all) $t>0$ %$t\ge \frac{q}{p}$
%we have
%\[
%\int_D |K(z,a)|^{\theta\frac{q}{p}+\frac{t}{n+1}}\,d\mu(z)\preceq \delta(a)^{-t}\;.
%\]
\end{itemize}
Moreover we have
\begin{equation}\label{norme}
%\begin{aligned}
%\[
\|\mu\|_{p,q;\alpha}%{A^p(D,\alpha)\hookrightarrow L^q(\mu)}&
\approx \|\hat\mu_{r,\frac{q}{p}\theta}\|_\infty\approx\|\hat\mu_{r,\frac{q}{p}}\delta^{-\alpha\frac{q}{p}}\|_\infty
%&\approx \sup_{z\in D} \frac{\mu\bigl(B_D(z,r_0)\bigr)}{\nu\bigl(B_D(z,r_0)\bigr)^{\frac{q}{p}\theta}}\\
%&\approx \sup_{k} \frac{\mu\bigl(B_D(a_k,r_0)\bigr)}{\nu\bigl(B_D(a_k,r_0)\bigr)^{\frac{q}{p}\theta}}\\
%\approx\sup_k \hat\mu_{r_0,\frac{q}{p}\theta}(a_k)
\approx \|\delta^{(n+1)\left(\frac{s}{2}-\theta\frac{q}{p}\right)}B^s\mu \|_\infty\;.
%\]
%\end{aligned}
\end{equation}
\end{teorema}

The geometric characterization of $(p,q;\alpha)$-skew Carleson measures when $p>q$ has a slightly different flavor.

\begin{teorema}[{\cite{HuLvZhu}*{Theorem~3.3}, \cite{AbaRai}*{Theorem~2.16}}]
\label{carthetaCarldue}
Let $D\Subset\C^n$ be a bounded strongly
pseudoconvex domain. Let $0< q< p<+\infty$ and $\alpha>-1$; put $\theta=1+\frac{\alpha}{n+1}$.
Then the following assertions are
equivalent:
\begin{itemize}
\item[(i)] $\mu$ is a $(p,q;\alpha)$-skew Carleson measure;
\item[(ii)] $\mu$ is a vanishing $(p,q,\alpha)$-skew Carleson measure;
% \item[(iii)] $\hat\mu_{r,\theta\frac{q}{p}}\in L^{\frac{p}{p-q}}(D);
\item[(iii)] $\hat\mu_r \delta^{-\alpha\frac{q}{p}}\in L^{\frac{p}{p-q}}(D)$ for some (and hence any) $r\in(0,1)$;
%\item[(iv)] $\hat\mu_{r,\theta}\in L^{\frac{p}{p-q}}(D,\alpha)$ for some (and hence any) $r\in(0,1)$;
%\item[(v)] $\hat\mu_{r,\theta\frac{q}{p}}\in L^{\frac{p}{p-q}}\bigl(D,-(n+1)\bigr)$ for some (and hence any) $r\in(0,1)$;
\item[(iv)] for some (and hence any) $r\in(0,1)$ and for some (and hence any) $r$-lattice $\{a_k\}$ in $D$ we have
$\{\hat\mu_{r,\theta\frac{q}{p}}(a_k)\}\in\ell^{\frac{p}{p-q}}$;
%\item[(vii)] for some (and hence any) $r\in(0,1)$ and for some (and hence any) $r$-lattice $\{a_k\}$ in $D$ we have
%$\{\hat\mu_{r}(a_k)\delta(a_k)^{(n+1)\left(1-\theta\frac{q}{p}\right)}\}\in\ell^{\frac{p}{p-q}}$;
%\item[(viii)] for some (and hence all) $s>\theta\frac{q}{p}+\frac{n}{n+1}\left(1-\frac{q}{p}\right)$ we have
%\[
%\delta^{-(n+1)\left(\theta\frac{q}{p}-\frac{s}{2}\right)}B^s\mu\in L^{\frac{p}{p-q}}\bigl(D,-(n+1)\bigr)\;;
%\]
%\item[(v)] for some (and hence all) $s>\theta\frac{q}{p}+\frac{n}{n+1}\left(1-\frac{q}{p}\right)$ we have
%\[
%\delta^{-(n+1)\left(\theta-\frac{s}{2}\right)}B^s\mu\in L^{\frac{p}{p-q}}(\nu_\alpha)\;;
%\]
\item[(v)] for some (and hence all) $s>\theta\frac{q}{p}+\frac{n}{n+1}\left(1-\frac{q}{p}\right)$ we have
\[
\delta^{-(n+1)\left(\theta\frac{q}{p}-\frac{s}{2}+\frac{p-q}{p}\right)}B^s\mu\in L^{\frac{p}{p-q}}(D)\;;
\]
%\item[(xi)] for some (and hence all) $t>(n+1)\left(1-\frac{q}{p}\right)\left(\frac{n}{n+1}-\theta\right)$ we have
%\[
%\delta^t \int_D |K(\cdot,w)|^{\theta+\frac{t}{n+1}}\,d\mu(w)\in L^{\frac{p}{p-q}}(D,\alpha)\;.
%\]
\end{itemize}
Moreover we have
\begin{equation}\label{norme2}
%\begin{aligned}
%\[
\|\mu\|_{p,q;\alpha}%A^p(D,\alpha)\hookrightarrow L^q(\mu)}&
\approx  \|\hat\mu_r\delta^{-\alpha\frac{q}{p}} \|_{\frac{p}{p-q}}
\approx \|\hat\mu_{r,\theta\frac{q}{p}}(a_k)\|_{\frac{p}{p-q}}
\approx \|\delta^{-(n+1)(\theta\frac{q}{p}-\frac{s}{2}+\frac{p-q}{p})}B^s\mu\|_{\frac{p}{p-q}}\;.
%&\simeq \|\delta^{-(n+1)\left(\theta\frac{q}{p}-\frac{t}{2}+\frac{p-q}{p}\right)}B^t\mu\|_{\frac{p}{p-q}}\\
%\approx \|\{\hat\mu_{r}(a_k)\delta(a_k)^{(n+1)\left(1-\theta\frac{q}{p}\right)}\}\|_{\frac{p}{p-q}}\;.
%\]
%\end{aligned}
\end{equation}
\end{teorema}

We also have a geometric characterization of vanishing $(p,q;\alpha)$-skew Carleson measures when $p\le q$:

\begin{teorema}[{\cite{HuLvZhu}*{Theorem~3.1}, \cite{AbaRaiSar}*{Theorem~4.10}}]
\label{carthetavanCarl}
Let $D\Subset\C^n$ be a bounded strongly
pseudoconvex domain. Let $0< p\le q<+\infty$ and $\alpha>-1$; set $\theta=1+\frac{\alpha}{n+1}$.
Then the following assertions are
equivalent:
\begin{itemize}
\item[(i)] $\mu$ is a vanishing $(p,q;\alpha)$-skew Carleson measure;
\item[(ii)] $\mu$ is a geometric vanishing $\frac{q}{p}\theta$-Carleson measure;
\item[(iii)] there exists $r_0\in(0,1)$ such that $\lim\limits_{z\to\partial D}\hat\mu_{r_0,\frac{q}{p}\theta}(z)=0$;
\item[(iv)] for some (and hence any) $r\in(0,1)$ we have $\lim\limits_{z\to\partial D}\hat\mu_{r,\frac{q}{p}}(z)\delta(z)^{-\alpha\frac{q}{p}}=0$;
%$\mu\bigl(B_D(\cdot,r_0)\bigr)\preceq \nu\bigl(B_D(\cdot,r_0)\bigr)^{\frac{q}{p}\theta}$;
\item[(v)] for some (and hence any) $r\in(0,1)$ and some (and hence any) $r$-lattice $\{a_k\}$ in $D$ we have
\[
\lim_{k\to+\infty} \hat\mu_{r_0,\frac{q}{p}\theta}(a_k)=\lim_{k\to+\infty}\hat\mu_{r,\frac{q}{p}}(a_k)\delta(a_k)^{-\alpha\frac{q}{p}}=0\;;
\]
%\item[(v)] there exists $r_0\in(0,1)$ and a $r_0$-lattice $\{a_k\}$ in $D$ such that
%\[
%\mu\bigl(B_D(a_k,r_0)\bigr)\preceq\nu\bigl(B_D(a_k,r_0)\bigr)^{\frac{q}{p}\theta}\;;
%\]
\item[(vi)] for some (and hence all) $s>\theta\frac{q}{p}$ we have
\[
\lim_{z\to\partial D} \delta(z)^{(n+1)\left(\frac{s}{2}-\theta\frac{q}{p}\right)}B^s\mu(z)=0 \;.
\]
%\item[(vii)] there exists $C>0$ such that for some (and hence all) $t>0$ %$t\ge \frac{q}{p}$
%we have
%\[
%\int_D |K(z,a)|^{\theta\frac{q}{p}+\frac{t}{n+1}}\,d\mu(z)\preceq \delta(a)^{-t}\;.
%\]
\end{itemize}
\end{teorema}

A consequence of these theorems is that the property of being $(p,q;\alpha)$-skew Carleson actually depends only on the quotient $q/p$ and on $\alpha$. We shall then introduce the following definition:

\begin{defin}
Take $\lambda$, $\alpha\in\R$. A finite positive Borel measure $\mu$ on $D$ is a \emph{$(\lambda,\alpha)$-skew Carleson measure} if
\begin{itemize}
\item[--] $\lambda\ge 1$ and $\hat\mu_{r,\lambda}\delta^{-\alpha\lambda}
\in L^\infty(D)$ for some (and hence any) $r\in(0,1)$, and we shall put $\|\mu\|_{\lambda,\alpha}=\|\hat\mu_{r,\lambda}\delta^{-\alpha\lambda}\|_\infty$; or,
\item[--] $\lambda<1$ and $\hat\mu_r\delta^{-\alpha\lambda}\in L^{\frac{1}{1-\lambda}}(D)$ for some (and hence any) $r\in(0,1)$, and we shall put $\|\mu\|_{\lambda,\alpha}=\|\hat\mu_{r}\delta^{-\alpha\lambda}\|_{\frac{1}{1-\lambda}}$.
\end{itemize}
Notice that by \cite{HuLvZhu}*{Lemma 2.3} different~$r$'s yield equivalent norms.

Furthermore we say that $\mu$ is \emph{vanishing $(\lambda,\alpha)$-skew Carleson measure} if
\begin{itemize}
\item[--] $\lambda\ge 1$ and $\lim\limits_{z\to\partial D}\hat\mu_{r,\lambda}(z)\delta(z)^{-\alpha\lambda}=0$
for some (and hence any) $r\in(0,1)$; or,
\item[--] $\lambda<1$ and $\mu$ is a $(\lambda,\alpha)$-skew Carleson measure.
\end{itemize}
So a measure is (vanishing) $(p,q;\alpha)$-skew Carleson if and only if it is (vanishing) $(q/p,\alpha)$-skew Carleson. Notice that
the definition of $(0,\alpha)$-skew Carleson does not depend on $\alpha$.
%Here
%$$
%\hat\mu_{r,\lambda}(z)=\frac{\mu\bigl(B_D(z,r)\bigr)}{\nu\bigl(B_D(z,r)\bigr)^\lambda}
%$$
%and $\hat\mu_r=\hat\mu_{r,1}$.
\end{defin}

\smallskip This definition has the following easy (but useful) consequence.

\begin{lemma}[{\cite{AbaRai}*{Lemma 2.18}}]
\label{th:last}
Let $D\Subset\C^n$ be a bounded strongly pseudoconvex domain, $\lambda>0$ and $\alpha>-1$. Let $\mu$ be a $(\lambda,\alpha)$-skew Carleson measure, and $\beta>-\lambda(\alpha+1)$. Then $\mu_\beta=\delta^{\beta}\mu$ is a $(\lambda, \alpha+\frac{\beta}{\lambda})$-skew Carleson measure with $\|\mu_\beta\|_{\lambda,\alpha+\frac{\beta}{\lambda}}\approx \|\mu\|_{\lambda,\alpha}$.
%In particular, if $\lambda\ge 1$ then $\mu_\beta$ is a geometric $(\lambda\gamma+\beta)$-Carleson measure.
\end{lemma}

We end this section by recalling the main result in \cite{AbaRai}, which gives a characterisation of $(\lambda,\gamma)$-skew Carleson measures on bounded strongly
pseudoconvex domain through products of functions in weighted Bergman spaces.

%%%
\begin{teorema}[\cite{AbaRai}*{Theorem~1.1}]
\label{thm:1.1 AbaRai}
Let $D\Subset\C^n$ be a bounded strongly
pseudoconvex domain, and let $\mu$ be a positive finite Borel measure on $D$. Fix an integer $k\geq 1$, and let $0<p_j,q_j<+\infty$ and $-1<\alpha_j<+\infty$ be given for $j=1,\ldots, k$. Set
\[
\lambda = \sum_{j=1}^k\frac{q_j}{p_j}\qquad\mathit{and}\qquad\gamma=\frac{1}{\lambda}\sum_{j=1}^k \frac{\alpha_j q_j}{p_j}\;.
\]
Then $\mu$ is a $(\lambda,\gamma)$-skew Carleson measure if and only if there exists $C>0$ such that
\begin{equation}
\int_D \prod_{j=1}^k |f_j(z)|^{q_j}\,d\mu(z)\le C \prod_{j=1}^k \|f_j\|_{p_j,\alpha_j}^{q_j}\label{eq:main}
\end{equation}
for any $f_j\in A^{p_j}_{\alpha_j}(D)$. %Furthermore, $C\approx \|\mu\|_{\lambda,\gamma}^{\sum_j q_j}$.
\end{teorema}

\section{Toeplits operators and skew Carleson measures on weighted Bergman spaces}

This section is devoted to the proof of our main Theorem~\ref{th:1.3}. We shall need the following preliminary result:
%\begin{defin}
%Take $\lambda$, $\gamma\in\R$. A finite positive Borel measure $\mu$ on $D$ is a \textit{$(\lambda,\gamma)$-skew Carleson measure} if
%\begin{itemize}
%\item[--] $\lambda\ge 1$ and $\hat\mu_{r,\lambda}\delta^{-\gamma\lambda}\in L^\infty(D)$ for some (and hence any) $r\in(0,1)$; or,
%\item[--] $\lambda<1$ and $\hat\mu_r\delta^{-\gamma\lambda}\in L^{\frac{1}{1-\lambda}}(D)$ for some (and hence any) $r\in(0,1)$.
%\end{itemize}
%In particular, the definition of $(0,\gamma)$-skew Carleson does not depend on $\gamma$.
%Here
%$$
%\hat\mu_{r,\lambda}(z)=\frac{\mu\bigl(B_D(z,r)\bigr)}{\nu\bigl(B_D(z,r)\bigr)^\lambda}
%$$
%and $\hat\mu_r=\hat\mu_{r,1}$.
%\end{defin}

\begin{lemma}
\label{lmm_dual}
Let $D\Subset\C^n$ be a bounded strongly
pseudoconvex domain. Let $1<p<+\infty$, $-1<\alpha$,~$\alpha'<+\infty$ and put
$$
\b=\frac{\alpha}{p}+\frac{\alpha'}{p'}\;,
$$
where $p'$ the conjugate exponent of $p$.
Then the functional
$$
(f,g)_\b=\int_Df(z)\overline{g(z)}d\nu_\b(z)
$$
is a duality pairing between $A^{p}_{\alpha}(D)$ and $A^{p'}_{\alpha'}(D)$, where $\nu_\beta=\delta^\beta\nu$.
\end{lemma}

\begin{proof}
The continuous dual of $A^{p}_{\alpha}(D)$ is $A^{p'}_{\alpha}(D)$, with the usual pairing
$$
\langle f,h\rangle=\int_{D}f(z)\overline{h(z)}d\nu_{\alpha}(z)\;.
$$
Therefore
$$
(f,g)_\b=\int_{D}f(z)\overline{g(z)}d\nu_\beta(z)=\int_D f(z)\overline{g(z)}\delta(z)^{\beta-\alpha}d\nu_{\alpha}(z)=\langle f, g\delta^{\beta-\alpha}\rangle
$$
is a duality pairing between $A^{p}_{\alpha}(D)$ and $A^{p'}_{\alpha'}(D)$ as soon as $h=g\delta^{\beta-\alpha}\in A^{p'}_{\alpha}(D)$, i.e., as soon as
$$
\int_D |g(z)|^{p'}\delta(z)^{(\beta-\alpha)p'}d\nu_{\alpha}(z)<+\infty
$$
which is true because the choice of $\beta$ yields $(\beta-\alpha)p'+\alpha=\alpha'$.
\end{proof}

%\begin{lemma}\label{lmm_boh} Given $r\in(0,1)$, there exists a constant $C_r$ such that for every $a, z\in D$ we have
%$$|K_\b(z,w)|\leq C_r|K_\b(z,a)|\qquad \forall w\in B_D(a,r)\;.$$\end{lemma}
%\begin{proof}
%\end{proof}
Now we can prove Theorem~\ref{th:1.3}:

\begin{teorema}\label{teo_toeplitz}
Let $D\Subset\C^n$ be a bounded strongly
pseudoconvex domain.
Let $0<p_1$,~$p_2<+\infty$ and $-1<\alpha_1$,~$\alpha_2<+\infty$. Suppose that $\beta\in\R$ satisfies
\begin{equation}\label{eq_brutta}
n+1+\b>n\max\left\{1,\dfrac{1}{p_j}\right\}+\frac{1+\alpha_j}{p_j}
\end{equation}
for $j=1$,~$2$.
Put
$$
\lambda=1+\frac{1}{p_1}-\frac{1}{p_2}
$$
and, if $\lambda\ne 0$, put
$$
\gamma=\frac{1}{\lambda}\left(\b+\frac{\alpha_1}{p_1}-\frac{\alpha_2}{p_2}\right)\;.
$$
Then for any positive Borel measure $\mu$ on $D$ the following statements are equivalent:
\begin{itemize}
\item[(i)] $T_\mu^\beta\colon A^{p_1}_{\alpha_1}(D)\to A^{p_2}_{\alpha_2}(D)$ continuously;
\item[(ii)] $\mu$ is a $(\lambda,\gamma)$-skew Carleson measure.
\end{itemize}
Moreover, one has
$$
\|T_{\mu}^\b\|_{A^{p_1}_{\alpha_1}(D)\to A^{p_2}_{\alpha_2}(D)}\approx \|\mu\|_{\lambda,\gamma}\;.
$$
\end{teorema}

%\begin{rem}
%The inequality \eqref{eq_brutta} for $j=2$ implies
%$$
%1+\beta>\frac{1+\alpha_2}{p_2}>\frac{1+\alpha_2}{p_2}-\frac{1+\alpha_1}{p_1}\;.
%$$
%Then it immediately follows that if $\lambda>0$ then $\gamma>-1$.
%\end{rem}

\begin{proof}
The proof is divided into several cases.
\smallskip

\noindent\textbf{(i)$\Rightarrow$(ii)} We consider two cases: $\lambda\geq 1$ and $\lambda< 1$.

\textsl{Case 1.} Assume $\lambda\geq 1$. Let $a\in D$ and consider $f_a=K_\b(\cdot,a)$. By \eqref{eq_brutta} with $j=1$, we get $(n+1+\beta)p_1>n+1+\alpha_1$, which is equivalent to $\alpha_1-\beta<(n+\beta+1)(p_1-1)$, so, by Theorem \ref{teo_stima}, for $a\in D$ we have that
\begin{equation}
\label{eq:tuno}
\|K_\b(\cdot, a)\|_{p_1,\alpha_1}^{p_1}\preceq \delta(a)^{n+1+\alpha_1-(n+1+\beta)p_1}\;;
\end{equation}
in particular, $f_a\in A^{p_1}_{\alpha_1}(D)$. We can then apply the Toeplitz operator to $f_a$ and consider the value of the resulting function for $z=a$:
\begin{equation}
\label{eq:tdue}
\begin{aligned}
T_\mu^\b f_a(a)&=\int\limits_{D}K_\b(a,w)f_a(w)d\mu(w)%=\int_D K_\b(a,w)K_\b(w,a)d\mu(w)\\
=\int_D|K_\b(a,w)|^2d\mu(w)\\
&\geq\int_{B_D(a,r)}|K_\b(a,w)|^2d\mu(w)\succeq\frac{\mu(B_D(a,r))}{\delta(a)^{2(n+\b+1)}}
\end{aligned}
\end{equation}
as soon as $a$ is close enough to $\partial D$, where, in the last inequality, we used Lemma \ref{piu}.

Moreover, by Lemma \ref{due} %\cite[Lemma~2.5]{AbaRaiSar}
\begin{equation}
\label{eq:ttre}
\begin{aligned}
T_\mu^\b f_a(a)
&=\bigl[|T_\mu^\b f_a(a)|^{p_2}\bigr]^{1/p_2}
\preceq \frac{1}{\nu\bigl(B_D(a,r)\bigr)^{1/p_2}}\left[\int_{B_D(a,r)}|T_\mu^\b f_a(\zeta)|^{p_2}d\nu(\zeta)\right]^{1/p_2}\\
&\preceq \frac{\delta(a)^{-\alpha_2/p_2}}{\nu\bigl(B_D(a,r)\bigr)^{1/p_2}}\left[\int_{B_D(a,r)}|T_\mu^\b f_a(\zeta)|^{p_2}\delta(\zeta)^{\alpha_2}d\nu(\zeta)\right]^{1/p_2}\\
&\preceq \delta(a)^{-(n+1+\alpha_2)/p_2}\|T_\mu^\beta f_a\|_{p_2,\alpha_2}\preceq \|T^\beta_\mu\|\delta(a)^{-(n+1+\alpha_2)/p_2}\|f_a\|_{p_1,\alpha_1}\;,
\end{aligned}
\end{equation}
where we used Lemma \ref{sei} and Lemma \ref{sette}.

%and, applying H\"older inequality to the functions $1$ and $|T_\mu^\b f_a(\zeta)|$, with conjugate exponents $\p2$ and $\p2'$, we obtain
%$$T_\mu^\b f_a(a)\leq K_r(\nu(B_D(a,r))^{\frac{1}{\p2'}-1}\left(\int\limits_{B_D(a,r)}|T_\mu^\b f_a(\zeta)|^{\p2}d\nu(\zeta)\right)^{\frac{1}{\p2}}\approx$$
%$$\approx \delta(a)^{-(n+1+\a2)/\p2}\left(\int\limits_{B_D(a,r)}|T_\mu^\b f_a(\zeta)|^{\p2}\delta(\zeta)^{\a2}d\nu(\zeta)\right)^{\frac{1}{\p2}}\precsim$$
%$$\precsim \delta(a)^{-(n+1+\a2)/\p2}\|T_\mu^\b f_a\|_{\p2,\a2}\approx \|T_\mu^\b\|\delta(a)^{(n+1+\a1)/\p1-(n+1+\a2)/\p2-(n+1+\b)}\;.$$

Combining \eqref{eq:tuno}, \eqref{eq:tdue} and \eqref{eq:ttre} we conclude that
\begin{equation}
\label{eq:tqua}
\begin{aligned}
\mu(B_D(a,r))&\preceq \|T_\mu^\b\|\delta(a)^{(n+1+\b)+(n+1+\a1)/\p1-(n+1+\a2)/\p2}=\|T_\mu^\b\|\delta(a)^{(n+1+\gamma)\lambda}\\
&\approx\|T_\mu^\b\|\nu\bigl(B_D(a,r)\bigr)^{(n+1+\gamma)\lambda/(n+1)}\;.
\end{aligned}
\end{equation}
This means that $\mu$ is a geometric $\lambda\left(1+\frac{\gamma}{n+1}\right)$-Carleson measure, which, by Theorem \ref{carthetaCarluno}, is equivalent to $\mu$ being a $(\lambda,\gamma)$-skew Carleson measure. Moreover,
$$
\|\mu\|_{\lambda,\gamma}\preceq\|T_\mu^\b\|\;.
$$

\smallbreak
%\begin{rem}
%In this argument we used the assumption
%$$
%n+1+\beta>n+\frac{1+\alpha_1}{p_1}\;.
%$$
%By Corollary~\ref{th:stimanorme}, we have $f_a\in A^{p_1}_{\alpha_1}$ also if this assumption does not hold, but with a different estimate on $\|f_a\|_{p_1,\alpha_1}$. Arguing as before we then get
%$$
%\mu\bigl(B_D(a,r)\bigr)\preceq
%\begin{cases}
%\nu\bigl(B_D(a,r)\bigr)^{2-\frac{1}{p_2}+\frac{2\beta-\alpha_2/p_2}{n+1}}\bigl|\log\nu(B_D(a,r)\bigr)\bigr|&$if $n+1+\beta=n+\frac{1+\alpha_1}{p_1}\\
%\nu\bigl(B_D(a,r)\bigr)^{2-\frac{1}{p_2}+\frac{2\beta-\alpha_2/p_2}{n+1}}&$if $n+1+\beta<n+\frac{1+\alpha_1}{p_1}\;.
%\end{cases}
%$$
%Thus $\mu$ is skew Carleson in these cases too.
%\end{rem}

\textsl{Case 2.} Assume $\lambda<1$, that is $\p2<\p1$. In this case, we can adapt the proof of \cite{AbaRai}*{Proposition 3.4} and we report here the complete proof for the sake of completeness.

Let $\{a_k\}$ be an $r$-lattice in $D$, and $\{r_k\}$ a sequence of Rademacher functions (see \cite{Duren}*{Appendix A}). %Take $\sigma\in\N$ such that $\sigma p>\max\{1,\theta_1\}$, put
Set
\[
\tau=\frac{n+1+\beta}{2}-\frac{n+1+\alpha_1}{p_1}\;,
\]
and, for every $a\in D$, put $f_a=\delta(a)^\tau k_{\beta, a}$. Then Lemma \ref{th:fact4} implies that
\[
f_t=\sum_{k=0}^\infty c_kr_k(t)f_{a_k}
\]
belongs to $A^{p_1}_{\alpha_1}(D)$ for all $\mathbf{c}=\{c_k\}\in\ell^{p_1}$, and $\|f_t\|_{p_1,\alpha_1}\preceq\|\mathbf{c}\|_{p_1}$.

Since, by assumption, $T^\beta_\mu$ is bounded from $A^{p_1}_{\alpha_1}$ to $A^{p_2}_{\alpha_2}$ we have
\begin{equation*}\begin{aligned}
\|T^\beta_\mu f_t\|^{p_2}_{p_2,\alpha_2}
&=\int_D\left|\sum_{k=0}^\infty c_k r_k(t) T^\beta_\mu f_{a_k}(z)\right|^{p_2}\;d\nu_{\alpha_2}(z)\\
&\le\|T^\beta_\mu\|^{p_2}\|f_t\|^{p_2}_{p_1,\alpha_1}\preceq \|T^\beta_\mu\|^{p_2}\|\mathbf{c}\|^{p_2}_{p_1}\;.
\end{aligned}\end{equation*}
Integrating both sides on $[0,1]$ with respect to $t$ and using Khinchine's inequality (see, e.g., \cite{Luecking}) we obtain
\[
\int_D\left(\sum_{k=0}^\infty |c_k|^2|T^\beta_\mu f_{a_k}(z)|^2\right)^{p_2/2}\;d\nu_{\alpha_2}(z)\preceq \|T^\beta_\mu\|^{p_2}\|\mathbf{c}\|^{p_2}_{p_1}\;.
\]
Set $B_k=B_D(a_k,r)$.
%Let $\{B_k\}$ be the sequence of Borel subsets of $D$ associated to $\{a_k\}$ given by \cite[Lemma~2.5]{AbaSar}.
We consider two cases: $p_2\ge 2$ and $0<p_2<2$.

If $p_2\ge 2$, using the fact that $\|\mathbf{a}\|_{p_2/2}\le\|\mathbf{a}\|_1$ for every $\mathbf{a}\in\ell^1$ we get
\begin{equation*}\begin{aligned}
\sum_{k=0}^\infty |c_k|^{p_2}\int_{B_k}|&T^\beta_\mu f_{a_k}(z)|^{p_2}\;d\nu_{\alpha_2}(z)\\
&\le
\int_D\left(\sum_{k=0}^\infty |c_k|^2|T^\beta_\mu f_{a_k}(z)|^2 \chi_{B_k}(z)\right)^{p_2/2}\;d\nu_{\alpha_2}(z)\\
&\le \int_D\left(\sum_{k=0}^\infty |c_k|^2|T^\beta_\mu f_{a_k}(z)|^2\right)^{p_2/2}\;d\nu_{\alpha_2}(z)\;.
\end{aligned}\end{equation*}
If instead $0<p_2<2$, using H\"older's inequality, we obtain
\begin{equation*}
\begin{aligned}
\sum_{k=0}^\infty |c_k|^{p_2}\int_{B_k}&|T^\beta_\mu f_{a_k}(z)|^{p_2}\;d\nu_{\alpha_2}(z)\\
&\le
\int_D\left(\sum_{k=0}^\infty |c_k|^2|T^\beta_\mu f_{a_k}(z)|^2\right)^{\frac{p_2}{2}}\left(\sum_{k=0}^\infty\chi_{B_k}(z)\right)^{1-\frac{p_2}{2}}\;d\nu_{\alpha_2}(z)\\
&\preceq\int_D\left(\sum_{k=0}^\infty |c_k|^2|T^\beta_\mu f_{a_k}(z)|^2\right)^{p_2/2}\;d\nu_{\alpha_2}(z)\;,
\end{aligned}
\end{equation*}
where we used the fact that each $z\in D$ belongs to no more than $m$ of the $B_k$.

Summing up, for any $p_2>0$ we have
\[
\sum_{k=0}^\infty |c_k|^{p_2}\int_{B_k}|T^\beta_\mu f_{a_k}(z)|^{p_2}\;d\nu_{\alpha_2}(z)\preceq \|T^\beta_\mu\|^{p_2}\|\mathbf{c}\|^{p_2}_{p_1}\;.
\]
Now Lemmas \ref{sei}, \ref{sette} and \ref{due} yield
\[
|T^\beta_\mu f_{a_k}(a_k)|^{p_2}
\preceq \delta(a_k)^{-(n+1+\alpha_2)} \int_{B_k} |T^\beta_\mu f_{a_k}(z)|^{p_2}\;d\nu_{\alpha_2}(z)\;,
\]
and so we have
\[
\sum_{k=0}^\infty |c_k|^{p_2}\delta(a_k)^{n+1+\alpha_2}|T^\beta_\mu f_{a_k}(a_k)|^{p_2}\preceq \|T^\beta_\mu\|^{p_2}\|\mathbf{c}\|^{p_2}_{p_1}\;.
\]
On the other hand, using Lemmas \ref{BKbasic} and \ref{piu}, we obtain
\begin{equation*}\begin{aligned}
T^\beta_\mu f_{a_k}(a_k)
&=\delta(a_k)^\tau \int_D K_\beta(a_k,w) k_{\beta,a_k}(w)\;d\mu(w)\\
&\succeq\delta(a_k)^{\tau+\frac{n+1+\beta}{2}}\int_D |K_\beta(a_k,w)|^{2} \;d\mu(w)\\
&\ge \delta(a_k)^{n+1+\beta-\frac{n+1+\alpha_1}{p_1}}\int_{B_D(a_k,r)} |K_\beta(a_k,w)|^{2} \;d\mu(w)\\
&\succeq \delta(a_k)^{n+1+\beta-\frac{n+1+\alpha_1}{p_1}}\frac{\mu\bigl(B_D(a_k,r)\bigr)}{\delta(a_k)^{2(n+1+\beta)}}=\frac{\mu\bigl(B_D(a_k,r)\bigr)}{\delta(a_k)^{n+1+\beta + \frac{n+1+\alpha_1}{p_1}}}\;.
\end{aligned}\end{equation*}
Putting all together we get
\[
\sum_{k=0}^\infty |c_k|^{p_2}\left(\frac{\mu\bigl(B_D(a_k,r)\bigr)}{\delta(a_k)^{(n+1+\gamma)\lambda}}\right)^{p_2}\preceq \|T^\beta_\mu\|^{p_2}\|\mathbf{c}\|^{p_2}_{p_1}\;,
\]
since
\[
n+1+\beta + \frac{n+1+\alpha_1}{p_1} - \frac{n+1+\alpha_2}{p_2} = (n+1+\gamma)\lambda\;.
\]
Now, set $\mathbf{d}=\{d_k\}$, where
\[
d_k=\frac{\mu\bigl(B_D(a_k,r)\bigr)}{\delta(a_k)^{(n+1+\gamma)\lambda}}\;.
\]
Then by duality we get $\{d_k^{p_2}\}\in\ell^{p_1/(p_1-p_2)}$ with $\|\{d_k^{p_2}\}\|_{p_1/(p_1-p_2)}\preceq \|T^\beta_\mu\|^{p_2}$, because $p_1/(p_1-p_2)$ is the conjugate exponent of $p_1/p_2>1$. This means that $\mathbf{d}\in\ell^{p_1p_2/(p_1-p_2)}=\ell^{1/(1-\lambda)}$ with
\[
\|\mathbf{d}\|_{\frac{1}{1-\lambda}}\preceq \|T^\beta_\mu\|\;,
\]
and the assertion then follows from Theorem \ref{carthetaCarldue} (notice that the proof in \cite{HuLvZhu} that $\{\hat\mu_{r,\lambda\theta}(a_k)\}\in\ell^{\frac{1}{1-\lambda}}$ implies $\hat\mu_r \delta^{-\lambda\gamma}\in L^{\frac{1}{1-\lambda}}(D)$, where $\theta=1+\frac{\gamma}{n+1}$, holds also for $\lambda\le 0$).

\bigskip

\noindent \textbf{(ii)$\Rightarrow$(i)} We consider three cases: $\p2>1$, $\p2=1$ and $0<\p2<1$.

\textsl{Case 1.} If $\p2>1$, let $\p2'>1$ be the conjugate exponent of $\p2$, and choose $\a2'\in\R$ so that
\begin{equation}\label{eq:sei}
\b=\frac{\a2}{\p2}+\frac{\a2'}{\p2'}\;.
\end{equation}
An easy computation shows that $\alpha'_2=\alpha_2+(\beta-\alpha_2)p_2'$, and then $\alpha'_2>-1$ follows from \eqref{eq_brutta} for $j=2$.

Take $f\in A^{p_1}_{\alpha_1}(D)$ and $h\in A^{p'_2}_{\alpha'_2}(D)$. Then
\begin{equation}
\begin{aligned}
(T_\mu^\b f, h)_\beta&=\int_{D}\overline{h(z)}\int_{D}K_\beta(z,w)f(w)d\mu(w)d\nu_\beta(z)\\
&=\int_{D}\overline{\int_{D}K_\beta(w,z)h(z)d\nu_\beta(z)}f(w)d\mu(w)=\int_D \overline{h(w)}f(w)d\mu(w)\;.
\end{aligned}
\label{eq:rep}
\end{equation}
Therefore, as $\mu$ is $(\lambda,\gamma)$-skew Carleson, by Theorem \ref{thm:1.1 AbaRai}, we have
$$
|(T_\mu^\b f, h)_\beta)|\preceq\|\mu\|_{\lambda,\gamma}\|f\|_{\p1,\a1}\|h\|_{\p2',\a2'}
$$
because, by our hypotheses,
$$
\lambda=\frac{1}{\p1}+\frac{1}{\p2'}\qquad\textrm{and}\qquad \gamma=\frac{1}{\lambda}\left(\frac{\a1}{\p1}+\frac{\a2'}{\p2'}\right)\;.
$$
As this holds for every $h\in A^{p'_2}_{\alpha'_2}(D)$, that is, by Lemma \ref{lmm_dual}, for every continuous functional on $A^{p_2}_{\alpha_2}(D)$,  we conclude that
$$
\|T_\mu^\b f\|_{\p2,\a2}\preceq \|\mu\|_{\lambda,\gamma}\|f\|_{\p1,\a1}\;,
$$
that is $T_{\mu}^\b$ is bounded from $A^{p_1}_{\alpha_1}(D)$ to $A^{p_2}_{\alpha_2}(D)$ and $\|T_\mu^\b\|\preceq\|\mu\|_{\lambda,\gamma}$.

\smallskip

\textsl{Case 2.} If $p_2=1$, that is $\lambda=\frac{1}{p_1}$, condition \eqref{eq_brutta} for $j=2$ implies $\beta-\a2>0$. Take $f\in A^{p_1}_{\alpha_1}(D)$. Then
$$
\begin{aligned}
\|T_\mu^\b f\|_{1,\a2}&\leq \int_{D}\int_{D}|K_\beta(z,w)||f(w)|d\mu(w)d\nu_{\a2}(z)\\
&=\int_{D}|f(w)|\int_{D}|K_\beta(z,w)|\delta(z)^{\a2-\b}d\nu_\b(z)\;d\mu(w)
\preceq \int_{D}|f(w)|\delta(w)^{\a2-\b}d\mu(w)
\end{aligned}
$$
by Theorem~\ref{teo_stima}.

Now, as $\mu$ is $(\lambda,\gamma)$-skew Carleson, Lemma \ref{th:last} implies that $\delta^{\alpha_2-\beta}\mu$ is $\left(\dfrac{1}{p_1},\alpha_1\right)$-skew Carleson, with $\|\delta^{\alpha_2-\beta}\mu\|_{1/p_1,\alpha_1}\approx\|\mu\|_{\lambda,\gamma}$. Theorems~\ref{carthetaCarluno} and~\ref{carthetaCarldue} then implies that $\delta^{\alpha_2-\beta}\mu$ is $(p_1,1;\alpha_1)$-skew Carleson, and so we obtain
$$
\|T_\mu^\beta f\|_{1,\alpha_2}\preceq\|\mu\|_{\lambda,\gamma}\|f\|_{p_1,\alpha_1}\;,
$$
as desired.

\smallskip

\textsl{Case 3.} If $0<\p2<1$, thanks to Lemma \ref{uno} we can find a $r$-lattice $\{a_k\}$ and $m\in\N$ such that for every $z\in D$ there exist at most $m$ values of $k$ such that $z\in B_D(a_k,R)$, where $R=\frac{1}{2}(1+r)$. Put $B_k=B_D(a_k,r)$ and $\widetilde{B}_k=B_D(a_k,R)$.

By Lemmas \ref{sei}, \ref{sette} and \ref{due} %and \ref{lmm_vol},
for $w\in B_k$ we have
$$
|f(w)|^{\p1}\preceq \frac{1}{\nu_{\a1}(B_k)}\int_{\widetilde{B}_k}|f(\zeta)|^{\p1}d\nu_{\a1}(\zeta)%\approx\frac{1}{\delta(a_k)^{n+1+\a1}}\int_{\widetilde{B}_k}|f(\zeta)|^{\p1}d\nu_{\a1}
$$
and
$$
|K_\b(z,w)|^{\p2}\preceq \frac{1}{\nu_{\a2}(B_k)}\int_{\widetilde{B}_k}|K_\b(z,\zeta)|^{\p2}d\nu_{\a2}(\zeta)\;.
$$
Therefore, integrating on $B_k$ we get
$$
\begin{aligned}
\int_{B_k}|K_\b(z,w)|&|f(w)|d\mu(w)\\
&\preceq \frac{\mu({B}_k)}{\nu_{\alpha_1}(B_k)^{1/p_1}\nu_{\alpha_2}(B_k)^{1/p_2}}\left(\int_{\widetilde{B}_k}|f(\zeta)|^{p_1}d\nu_{\alpha_1}(\zeta)\right)^{1/p_1}\left(\int_{\widetilde{B}_k}|K_\b(z,\zeta)|^{p_2}d\nu_{\alpha_2}(\zeta)\right)^{1/p_2}\;.
\end{aligned}
$$
Since $\p2<1$, summing over $k$ we get
$$
|T_\mu^\b f(z)|^{\p2}\preceq\sum_{k=1}^\infty\frac{\mu({B}_k)^{\p2}}{\nu_{\alpha_1}(B_k)^{p_2/\p1}\nu_{\alpha_2}(B_k)}\left(\int_{\widetilde{B}_k}|f(\zeta)|^{\p1}d\nu_{\a1}(\zeta)\right)^{\frac{\p2}{\p1}}\int_{\widetilde{B}_k}|K_\b(z,\zeta)|^{\p2}d\nu_{\a2}(\zeta)\;.
$$
Integrating in $z$ over $D$ with respect to $\nu_{\a2}$ we obtain
\begin{equation}
\label{eq_appezzi}
\|T_\mu^\b f\|_{\p2,\a2}^{\p2}\preceq \sum_{k=1}^\infty\frac{\mu({B}_k)^{\p2}}{\delta(a_k)^{(n+1+\gamma)\lambda\p2}}\left(\int_{\widetilde{B}_k}|f(\zeta)|^{\p1}d\nu_{\a1}(\zeta)\right)^{\frac{\p2}{\p1}}\;,
\end{equation}
thanks to Lemma~\ref{sei}, Lemma \ref{sette} and Theorem \ref{teo_stima}, that we can apply because of \eqref{eq_brutta} for $j=2$.

Now, if $\lambda\geq 1$ we have that
$$
\mu\left({B}_k\right)\preceq \|\mu\|_{\lambda,\gamma}\delta(a_k)^{(n+1+\gamma)\lambda}\;,
$$
and so \eqref{eq_appezzi} yields
$$
\|T_\mu^\b f\|_{\p2,\a2}^{\p2}\preceq\|\mu\|_{\lambda,\gamma}^{p_2}\sum_{k=1}^\infty\left(\int_{\widetilde{B}_k}|f(\zeta)|^{p_1}d\nu_{\alpha_1}(\zeta)\right)^{\frac{p_2}{p_1}}\preceq\|\mu\|_{\lambda,\gamma}^{\p2}\|f\|_{\p1,\a1}^{\p2}\;.$$

On the other hand, if $\lambda<1$ (that is $p_1/p_2>1$), by H\"older inequality we have
$$
\begin{aligned}
\sum_{k=1}^\infty\frac{\mu({B}_k)^{\p2}}{\delta(a_k)^{(n+1+\gamma)\lambda\p2}}&\left(\int_{\widetilde{B}_k}|f(\zeta)|^{\p1}d\nu_{\a1}(\zeta)\right)^{\frac{\p2}{\p1}}\\
&\!\!\!\!\leq\left(\sum_{k=1}^\infty\left(\frac{\mu({B}_k)}{\delta(a_k)^{(n+1+\gamma)\lambda}}\right)^\frac{\p1\p2}{\p1-\p2}\right)^\frac{\p1-\p2}{\p1}\left(\sum_{k=1}^\infty\int_{\widetilde{B}_k}|f(\zeta)|^{\p1}d\nu_{\a1}(\zeta)\right)^{\frac{\p2}{\p1}}\;.
\end{aligned}
$$
Now, the proof of the implication (b)$\Rightarrow$(c) in \cite{HuLvZhu}*{Lemma 2.5} applied with $s=-\gamma\lambda$ and $p=p_1p_2/(p_1-p_2)$ yields
%\cite[Theorem 2.16]{AbaRai} yields
$$
%\{\hat{\mu}_{\lambda\theta,r}(a_k)\}_{k\geq1}=
\left\{\frac{\mu({B}_k)}{\delta(a_k)^{(n+1+\gamma)\lambda}}\right\}_{k\geq1}\in \ell^{\frac{p_1p_2}{p_1-p_2}}
$$
%where $\theta=(n+1+\gamma)/(n+1)$ and $\nu(B_D(a_k,r))^{\theta\lambda}\approx \delta(a_k)^{(n+1+\gamma)\lambda}$; moreover
and
$$
\left\|\left\{\frac{\mu({B}_k)}{\delta(a_k)^{(n+1+\gamma)\lambda}}\right\}_{k\geq1}\right\|_{\ell^{\frac{p_1p_2}{p_1-p_2}}}\approx\|\mu\|_{\lambda,\gamma}\;.
$$
So
$$
\|T_\mu^\b f\|_{p_2,\alpha_2}^{p_2}\preceq \|\mu\|_{\lambda,\gamma}^{p_2}\|f\|_{p_1,\alpha_1}^{p_2}\;,
$$
and we are done in this case too.
\end{proof}

%\begin{rem}
%\textit{In the last step of the proof we have used the following definition of $(\lambda,\gamma)$-skew Carleson for $\lambda\le0$: a measure $\mu$ is $(\lambda,\gamma)$-skew Carleson if $\hat\mu_r\delta^{-\gamma\lambda}\in L^{\frac{1}{1-\lambda}}(D)$. By the way, this is equivalent to}
%$$
%\frac{\mu\bigl(B_D(\cdot,r)\bigr)}{\nu_\gamma\bigl(B_D(\cdot,r)\bigr)}\in L^{\frac{1}{1-\lambda}}(D,\gamma)\;.
%$$
%In particular, the definition of $(0,\gamma)$-skew Carleson is independent of $\gamma$.
%\end{rem}

\section{Compact Toeplitz operators and vanishing skew Carleson measures}

In this section we shall prove a version of Theorem~\ref{teo_toeplitz} concerning compact Toeplitz operators and vanishing skew-Carleson measures. The only interesting case is $\lambda\ge 1$, because for $\lambda<1$ (that is $p_2<p_1$) all $(\lambda,\gamma)$-skew Carleson measures are vanishing (Theorem~\ref{carthetaCarldue}) and all continuous operators from $A^{p_1}_{\alpha_1}(D)$ to $A^{p_2}_{\alpha_2}(D)$ are compact (see, e.g., \cite{LiTz}*{Proposition 2.c.3}).

To deal with the case $\lambda\ge 1$ we shall need the following version of Theorem~\ref{thm:1.1 AbaRai}, whose proof is analogous to the proof of \cite{PZ}*{Theorem~4.1}:

\begin{teorema}\label{teo_vanish_prod}
Let $D\Subset\C^n$ be a bounded strongly
pseudoconvex domain, and let $\mu$ be a positive finite Borel measure on $D$. Fix an integer $k\geq 1$, and let $0<p_j$,~$q_j<+\infty$ and $-1<\alpha_j<+\infty$ be given for $j=1,\ldots, k$. Set
\[
\lambda = \sum_{j=1}^k\frac{q_j}{p_j}\qquad\mathit{and}\qquad\gamma=\frac{1}{\lambda}\sum_{j=1}^k \frac{\alpha_j q_j}{p_j}\;.
\]
Assume that $\lambda\ge 1$. Then the following statements are equivalent:
\begin{itemize}
\item[(i)] $\mu$ is a vanishing $(\lambda,\gamma)$-skew Carleson measure.
\item[(ii)]  For any sequence $\{f_{1,\ell}\}_{\ell}$ in the unit ball of $A^{p_1}_{\alpha_1}(D)$ converging to $0$ uniformly on compact sets in~$D$ we have
\[
\lim_{\ell\to\infty}F(\ell)=0\;,
\]
where
\[
F(\ell)=\sup\left\{\int_{D}|f_{1,\ell}(z)|^{q_1}\prod_{j=2}^k|f_j(z)|^{q_j}d\mu(z)\biggm| \|f_j\|_{p_j,\alpha_j}\leq 1,\; j=2,\ldots, k\right\}\;.
\]
\item[(iii)] For any $k$ sequences $\{f_{1,\ell}\},\ldots, \{f_{k,\ell}\}$ in the unit balls of $A^{p_1}_{\alpha_1}(D),\ldots, A^{p_k}_{\alpha_k}(D)$, respectively, which are all convergent to $0$ uniformly on compact sets in $D$, we have
\[
\lim_{\ell\to\infty}\int_{D}|f_{1,\ell}(z)|^{q_1}\cdots|f_{k,\ell}(z)|^{q_k}d\mu(z)=0\;.
\]
\end{itemize}
\end{teorema}

\begin{proof}
%Thanks to Theorem \ref{carthetaCarldue}, the case $0<\lambda<1$ is a consequence of Theorem \ref{thm:1.1 AbaRai}. We therefore focus our attention on the case $\lambda\geq 1$.
%
Assume (i) is satisfied, that is $\mu$ is a vanishing $(\lambda,\gamma)$-skew Carleson measure.
Let $\{f_{1,\ell}\}_{\ell\in\N}$ be a sequence in the unit ball of $A^{p_1}_{\alpha_1}(D)$ which converges to $0$ uniformly on compact subsets of $D$, and for $j=2,\ldots, k$ let $f_j$ be an arbitrary function in the unit ball of $A^{p_j}_{\alpha_j}(D)$. Given $r>0$, let us set $D_r=\{z\in\ D\mid \delta(z)<r\}$. Then $\mu_r=\mu\vert_{D_r}$ is a $(\lambda,\gamma)$-skew Carleson measure, and
\[
\lim_{r\to 0}\|\mu_r\|_{\lambda,\gamma}= 0
\]
because $\mu$ is vanishing. Fix $\varepsilon>0$. Then if $r>0$ is small enough Theorem~\ref{thm:1.1 AbaRai} yields
\begin{equation}
%\begin{aligned}
\int_{D_r}\!\!\!\!|f_{1,\ell}(z)|^{q_1}\!|f_2(z)|^{q_2}\!\cdots |f_k(z)|^{q_k}d\mu(z)
= \int_{D}\!\!\!|f_{1,\ell}(z)|^{q_1}\!|f_2(z)|^{q_2}\!\cdots |f_k(z)|^{q_k}d\mu_r(z)\preceq \varepsilon\;.
%&\preceq \|\mu_r\|_{\lambda,\gamma}\leq C\varepsilon\;,
%\end{aligned}
\end{equation}
On the other hand, thanks to the uniform convergence of $f_{1,\ell}$ to $0$ on compact subsets of $D$, we can find $M\in\N$ such that for any $\ell>M$ we have $|f_{1,\ell}(z)|<\varepsilon$ for all $z\in D\setminus {D_r}$. Therefore applying again Theorem \ref{thm:1.1 AbaRai} we have
\begin{equation}
\begin{aligned}
\int_{D\setminus {D_r}}\!\!\!\!\!\!\!\!\!\!\!\!|f_{1,\ell}(z)|^{q_1}\!|f_2(z)|^{q_2}\cdots |f_k(z)|^{q_k}d\mu(z)
&\!\leq \varepsilon \!\!\int_{D}|f_2(z)|^{q_2}\cdots |f_k(z)|^{q_k}d\mu(z)\\
&=\!\varepsilon \!\!\int_{D}|1|^{q_1}|f_2(z)|^{q_2}\cdots |f_k(z)|^{q_k}d\mu(z)%\\
%&\preceq \varepsilon\|1\|^{q_1}_{p_1,\alpha_1}\|f_2\|^{q_2}_{p_2,\alpha_2}\cdots\|f_k\|^{q_k}_{p_k,\alpha_k}
\preceq\varepsilon\;.
\end{aligned}
\end{equation}
These last two estimates together imply (ii).

It is evident that (ii) implies (iii).
To prove that (iii) implies (i) we follow the same construction as in the proof of Theorem \ref{thm:1.1 AbaRai}. Choose $\sigma_1, \dots, \sigma_k\in\N^*$ such that
\[
p_j\sigma_j>\max\left\{1,1+ \frac{\alpha_j}{n+1}\right\}
\]
for all $j=1,\ldots,k$, and
\[
\sum_{j=1}^k q_j\sigma_j > \lambda\gamma\;,
\]
and set
\[
r_j=\frac{(n+1)\sigma_j}{2}-\frac{n+1+\alpha_j}{p_j}\;.
\]
For any $a\in D$ and $j=1,\ldots,k$, consider
\[
f_{j,a}(z)=\delta(a)^{r_j} k_a(z)^{\sigma_j}\;.
\]
Then, since $\alpha_j<(n+1)(p_j\sigma_j-1)$ by the choice of~$\sigma_j$ we know (Theorem~\ref{teo_stima}) that $\|f_{j,a}\|_{p_j,\alpha_j}\preceq 1$ for all $j=1,\dots, k$; moreover it is easy to see that
\[
\lim_{a\to\partial D} |f_{j,a}(z)|= 0
\]
uniformly on any compact subset of $D$. Therefore (iii) yields
\begin{equation}
\lim_{a\to\partial D} \int_D\prod_{j=1}^k |f_{j,a}(z)|^{q_j}\;d\mu(z) = 0\;.
\label{eq:nove}
\end{equation}
Now, we have
\[
\int_D\prod_{j=1}^k |f_{j,a}(z)|^{q_j}\;d\mu(z)
=\int_D|k_a(z)|^{\sum_j q_j\sigma_j}\delta(a)^{\sum_j q_jr_j}\;d\mu(z),
\]
and
\[
\sum_{j=1}^k q_jr_j=(n+1)\sum_{j=1}^k\left[\frac{q_j\sigma_j}{2}-\theta_j\frac{q_j}{p_j}\right]
=\frac{n+1}{2}\sum_{j=1}^k q_j\sigma_j-(n+1)\lambda\gamma\;.
\]
Therefore, setting $s=\sum_j \sigma_jq_j>\lambda\gamma$, \eqref{eq:nove} becomes
\[
\lim_{a\to\partial D}\delta(a)^{(n+1)\left(\frac{s}{2}-\lambda\gamma\right)}\int_D|k_a(z)|^s\;d\mu(z)
=
\lim_{a\to\partial D} \delta(a)^{(n+1)\left(\frac{s}{2}-\lambda\gamma\right)}B^s\mu(a)=0,
\]
where $B^s\mu$ is the Berezin transform of level $s$ of $\mu$, and so $\mu$ is a vanishing $(\lambda,\gamma)$-skew Carleson measure thanks to Theorem~\ref{carthetavanCarl}.
\end{proof}

We can now prove the following result:

\begin{teorema}\label{teo_toeplitz_vanishing}
Let $D\Subset\C^n$ be a bounded strongly
pseudoconvex domain.  Let $0<p_1\le p_2<+\infty$ and $-1<\alpha_1$,~$\alpha_2<+\infty$. Suppose that $\beta\in\R$ satisfies
\begin{equation}\label{eq_brutta_vanishing}
n+1+\b>n\max\left\{1,\dfrac{1}{p_j}\right\}+\frac{1+\alpha_j}{p_j}
\end{equation}
for $j=1$,~$2$. Put
$$
\lambda=1+\frac{1}{p_1}-\frac{1}{p_2}
$$
and
$$
\gamma=\frac{1}{\lambda}\left(\beta+\frac{\alpha_1}{p_1}-\frac{\alpha_2}{p_2}\right)\;.
$$
Then for any positive Borel measure $\mu$ on $D$ the following statements are equivalent:
\begin{itemize}
\item[(i)] $T_\mu^\beta\colon A^{p_1}_{\alpha_1}(D)\to A^{p_2}_{\alpha_2}(D)$ compactly;
\item[(ii)] $\mu$ is a vanishing $(\lambda,\gamma)$-skew Carleson measure.
\end{itemize}
\end{teorema}

\begin{proof}
%Thanks to Theorem \ref{carthetaCarldue}, the case $0<\lambda<1$ is trivial. In fact, in this case every $(\lambda,\gamma)$-skew Carleson measure is vanishing and $T_\mu^\b$ is compact if and only if it is bounded because $p_2<p_1$ (see, e.g., \cite{LiTz}*{Proposition 2.c.3}). %If needed, one can put here the same sentence as in Pau-Zhao Carleson Measures and Toeplitz operators for weighted Bergman spaces on the unit ball Theorem 4.2. The same proof works here, and one can use Proposition 2.c.3 in the book of Lindenstrauss-Tzafriri "Classical Banach spaces. I and II" Springer 1996

Assume that (i) holds. Since $T_\mu^\b$ is compact, it maps every bounded sequence in $A^{p_1}_{\alpha_1}(D)$ converging uniformly to $0$ on compact subsets of $D$ to a sequence strongly converging to $0$ in $A^{p_2}_{\alpha_2}(D)$.

We consider a sequence $\{a_k\}\in D$ such that $\lim\limits_{k\to+\infty}\delta(a_k)=0$ and we set
\[
f_k(z)
=
\delta(a_k)^{(n+1+\b)-(n+1+\a1)/\p1}K_\b(z,a_k)\;.
\]
Thanks to Theorem \ref{teo_stima}, we have that
\[
\|f_k\|_{\p1,\a1}^{\p1}\preceq 1\;.
\]
Moreover, for any $L\Subset D$ there exists a constant $C_1>0$ such that $|K_\beta|$ is bounded from above by $C_1$ on $L\times D$. Therefore for every $z\in L$ we have that
\[
|f_k(z)|\leq C_1\delta(a_k)^{n+1+\b - (n+1+\a1)/\p1}
\]
and so, since since our hypotheses give us that $(n+1+\beta)-(n+1+\a1)/\p1>0$, we get
\[
\lim_{k\to+\infty} \sup_{z\in L}|f_k(z)|\le  \lim_{k\to+\infty} C_1\delta(a_k)^{n+1+\b - (n+1+\a1)/\p1} =0\;.
\]
Hence the compactness of $T_\mu^\b$ implies
\begin{equation}\label{eq:11}
\lim_{k\to+\infty} \|T_\mu^\b f_k\|_{\p2,\a2}= 0\;.
\end{equation}
Now, the same computations as in the proof of the implication (i)$\Longrightarrow$(ii) of Theorem \ref{teo_toeplitz} yield
\[
\frac{\mu(B_D(a_k,r))}{\delta(a_k)^{(n+1+\b)+(n+1+\a1)/\p1}}
\preceq
T_\mu^\b f_k(a_k)
\]
and
\[
T_\mu^\b f_k(a_k)
\preceq
\delta(a_k)^{-(n+1+\a2)/\p2}\|T_\mu^\b f_k\|_{\p2,\a2}\;.
\]
Therefore
\[
\frac{\mu(B_D(a_k,r))}{\delta(a_k)^{(n+1+\gamma)\lambda}}\preceq\|T_\mu^\b f_k\|_{\p2,\a2}\;,
\]
which, together with \eqref{eq:11} and Theorem~\ref{carthetavanCarl}, implies that $\mu$ is a vanishing $(\lambda,\gamma)$-skew Carleson measure.

Conversely,
assume that $\mu$ is a vanishing $(\lambda,\gamma)$-skew Carleson measure with $\lambda\geq 1$, and let $\{g_k\}_{k\in\N}$ be a bounded sequence in $A^{p_1}_{\alpha_1}(D)$ converging uniformly to $0$ on compact subsets of $D$. We want to prove that the bounded sequence $\{T_\mu^\beta g_k\}_{k\in\N}\subset A^{p_2}_{\alpha_2}(D)$ converges strongly to $0$ in $A^{p_2}_{\alpha_2}(D)$. We consider two cases: $\p2>1$ and $0<\p2\le 1$.

If $\p2>1$ then, as in the proof of Theorem~\ref{teo_toeplitz}, thanks to Lemma \ref{lmm_dual}, denoting by $\p2'$ the conjugate exponent of $\p2$ and $\a2'$ the number defined in \eqref{eq:sei}, using \eqref{eq:rep}we have
\[
\|T_\mu^\b g_k\|_{\p2,\a2}
\approx
\sup_{\|h\|_{\p2',\a2'}\leq 1}|\langle h, T_\mu^\b g_k\rangle_\b|
\leq
\sup_{\|h\|_{\p2',\a2'}\leq 1}\int_{D}|h(z)||g_k(z)|d\mu(z)\;,
\]
and Theorem \ref{teo_vanish_prod} yields that the last integral converges to $0$ as $k$ tends to $+\infty$.

If $0<\p2\leq 1$, for any $r$-lattice $\{a_j\}$ we consider the associated balls $\{B_j=B_D(a_j,r)\}$ and $\{\widetilde B_j=B_D(a_j,R)\}$, where $R=(1+r)/2$, as usual. Using \eqref{eq_appezzi} we obtain that
\begin{equation}\label{eq:13}
\|T_\mu^\b g_k\|_{\p2,\a2}^{\p2}
\preceq
\sum_{j=1}^\infty \frac{\mu\left({B}_j\right)^{\p2}}{\delta(a_j)^{(n+1+\gamma)\lambda\p2}} \left(\int_{\widetilde{B}_j}|g_k(\zeta)|^{\p1} d\nu_{\a1}(\zeta)\right)^{{\p2}/{\p1}}\;.
\end{equation}
Let $\varepsilon>0$. Since $\mu$ is a vanishing $(\lambda,\gamma)$-skew Carleson measure by  Theorem~\ref{carthetavanCarl} there exists $j_0>0$
such that
\[
\frac{\mu({B}_j)}{\delta(a_j)^{(n+1+\gamma)\lambda}}<\varepsilon
\]
for all $j>j_0$. Choose $\delta_0>0$ such that $\tilde B_j\subset L=\{z\in D\mid \delta(z)\geq\delta_0\}\Subset D$ for all $j\le j_0$. We can then split the sum in the right-hand-side of \eqref{eq:13} into two parts. For the first part we have
\[
\sum_{j=1}^{j_0}\frac{\mu\left({B}_j\right)^{\p2}}{\delta(a_j)^{(n+1+\gamma)\lambda\p2}} \left(\int_{\widetilde{B}_j}|g_k(\zeta)|^{\p1} d\nu_{\a1}(\zeta)\right)^{\frac{\p2}{\p1}}
\!\!\preceq
\left(\sup_{L}|g_k|^{\p1}\right)^{\frac{\p2}{p1}} \!\sum_{j=1}^{j_0}
\frac{\mu\left({B}_j\right)^{\p2}\nu_{\alpha_1}({B}_j)^{\frac{\p2}{\p1}}}{\delta(a_j)^{(n+1+\gamma)\lambda\p2}}
\]
and clearly the right-hand-side converges to $0$ as $k$ tends to $+\infty$.

On the other hand we have
\begin{equation*}
\begin{aligned}
\sum_{j=j_0+1}^\infty \frac{\mu\left({B}_j\right)^{\p2}}{\delta(a_j)^{(n+1+\gamma)\lambda\p2}} \left(\int_{\widetilde{B}_j}|g_k(\zeta)|^{\p1} d\nu_{\a1}(\zeta)\right)^{\frac{\p2}{\p1}}
&<
\varepsilon^{\p2}\sum_{j=j_0+1}^\infty \left(\int_{\widetilde{B}_j}|g_k(\zeta)|^{\p1} d\nu_{\a1}(\zeta)\right)^{\frac{\p2}{\p1}}\\
&\preceq
\varepsilon^{\p2}\|g_k\|_{\p1,\a1}^{\p2}\preceq \varepsilon^{\p2}
\end{aligned}
\end{equation*}
because the sequence~$\{g_k\}$ is norm-bounded. Therefore $\lim\limits_{k\to+\infty}\|T_\mu^\b g_k\|_{\p2,\a2}= 0$, and this concludes the proof.
\end{proof}

\end{document}